\documentclass[11pt]{article}
\usepackage{amssymb,enumerate}
\usepackage{amsfonts}
\usepackage{amsmath}
\usepackage{cases}
\usepackage[usenames]{color}
\usepackage{mathrsfs}
\usepackage{cite}
\usepackage{cases}
\usepackage{amsthm}

\usepackage{indentfirst}
\usepackage{setspace}
\usepackage{abstract}
\usepackage[none]{hyphenat}
\usepackage[parfill]{parskip}

\allowdisplaybreaks[4]
\def\vs{\vspace*}

\def\Z{\mathbb{Z}}
\def\N{\mathbb{N}}

\def\C{\mathbb{C}}

\def\C{\mathbb{C}}

\def\L{\mathfrak{L}}

\def\SUM#1#2{{\mbox{$\sum\limits_{#1}^{#2}$}}}

\numberwithin{equation}{section}
\newtheorem{theo}{Theorem}[section]
\newtheorem{defi}[theo]{Definition}

\newtheorem{lemm}[theo]{Lemma}

\newtheorem{prop}[theo]{Proposition}
\newtheorem{clai}{Claim}

\begin{document}
\sloppy{}
\baselineskip 6pt
\lineskip 6pt

\title{\bf Induced  modules and central character quotients 
for Takiff $\mathfrak{sl}_{2}$}

\author{Xiaoyu Zhu\\
{\footnotesize School of Mathematics and Statistics, Ningbo University, Ningbo, Zhejiang, 315211, China}\\
{\footnotesize E-mail: zhuxiaoyu1@nbu.edu.cn}}
\date{}
\maketitle

\noindent{{\bf Abstract.}} 
We construct a large new family of simple modules over
Takiff $\mathfrak{sl}_{2}$. We prove that the quotient of 
the universal enveloping algebra of the Takiff Lie algebra for $\mathfrak{sl}_{2}$
by the ideal generated by a non-trivial central character is a simple algebra.
In the case of the trivial central character, we show that the corresponding
ideal is primitive by explicitly constructing a simple module whose annihilator
coincides with that ideal. Together with the annihilators of simple 
$\mathfrak{sl}_{2}$-modules, we expect that the above ideals exhaust all
primitive ideal.
\vs{6pt}


\section{Introduction}

Classification and construction of modules are two  fundamental 
challenges in contemporary representation theory.
In representation theory of Lie algebras,  achieving a complete 
classification of all modules is generally considered to be a very difficult 
(perhaps even impossible) task.
It is thus natural to  narrow down the scope of the classification problem 
by focusing on "smaller" specific classes of modules.

Simple modules form one of the most natural general classes of modules.
Unfortunately, in the case of Lie algebras, 
except for the Lie algebra $\mathfrak{sl}_{2}$ (see \cite{RB,VM}),
a solution to the problem of classification of simple modules 
seems to be far out of reach for the moment. 
This motivates the study of general rough invariants of simple modules.
The annihilator of a simple module, 
which is a primitive ideal of the universal enveloping algebra,
is one of the most natural such invariants.
The primitive spectrum of the universal enveloping algebra 
for semisimple finite dimensional Lie algebras
has been extensively studied. In particular, in \cite{Du}, 
Duflo proved that every primitive ideal of a reductive Lie algebra corresponds 
to the annihilator ideal of a simple highest weight module.
An explicit classification of primitive ideals and their
inclusions in the semi-simple case was completed in 
\cite{BV1,BV2} based on the previous work by many people,
see \cite{Di,VJ} and \cite[Chapters 5 and 14]{{CJ}} for details.

For non-semi-simple Lie algebras,
the situation is much less studied than
in the case of semi-simple Lie algebras,
see e.g. \cite{BL, BL1, DB} for the results in some special  cases. 
There are several natural families of non-semi-simple Lie algebras, 
which include current Lie algebras, conformal Galilei algebras, 
Takiff Lie algebras, and others. 
Takiff Lie algebras can be seen as the "smallest" 
non-semi-simple truncated current Lie algebras, see \cite{BJW}. 
The Takiff Lie algebra for $\mathfrak{sl}_{2}$ also falls into 
the category of conformal Galilei algebras, see e.g. \cite{LMZ}, and is
defined as the semidirect product of $\mathfrak{sl}^{}_{2}$ with its adjoint representation. 

These particular Lie algebras were initially introduced by Takiff in \cite{T},
whose primary motivation was the invariant theory for such Lie algebras. 
Representation theory of Takiff Lie algebras and, especially, questions
related to highest weight modules  and BGG category $\mathcal{O}$ have 
recently attracted considerable attention, see \cite{MC,MM,MS,Zhu} and the references therein.
Our long-term project is to better understand representation  theory of  
Takiff Lie algebras, in particular, to understand primitive ideals
of the corresponding universal  enveloping algebra, for instance,
from  the point of view of Duflo's theorem, and to understand the
corresponding primitive quotients  and their representations.
In this paper, we  start by looking at this problem in the 
smallest possible  case of the Takiff Lie algebra for $\mathfrak{sl}_{2}$.

We obtain the following results: we construct a large new family of 
simple modules over Takiff $\mathfrak{sl}_{2}$. We prove that the quotient of 
the universal enveloping algebra of the Takiff Lie algebra for $\mathfrak{sl}_{2}$
by the ideal generated by a non-trivial central character is a simple algebra.
In the case of the trivial central character, we show that the corresponding
ideal is primitive by explicitly constructing a simple module whose annihilator
coincides with that ideal. Together with the annihilators of simple 
$\mathfrak{sl}_{2}$-modules, we expect that the above ideals exhaust all
primitive ideal.

This paper is organized as follows.
In Section 2, we introduce the basic setup that we work in 
and recall some necessary definitions and preliminary results.
In Section 3, we study Takiff $\mathfrak{sl}_{2}$ modules induced from 
simple highest weight $\mathfrak{sl}_{2}$-modules. 
In particular, in Proposition \ref{prop3.2} we show that generic
central  character quotients  of these induced modules are simple.
This produces a large new family of simple modules over
the Takiff $\mathfrak{sl}_{2}$. After that, in Theorem~\ref{lemma1} 
we prove that the annihilators of these simple modules  are generated 
by their central characters. In Section~4 we look at the central character
quotients of the universal enveloping algebra of the Takiff $\mathfrak{sl}_{2}$.
In Theorem~\ref{theo4.2} we show that the quotient by a
non-trivial central character is a simple algebra.
In Theorem~\ref{theo4.3} we show that the quotient by a
trivial central character is a primitive ideal.
This implies that all central character quotients of 
the universal enveloping algebra of the Takiff $\mathfrak{sl}_{2}$
are  indeed primitive ideals and we expect that,
together with the annihilators of simple 
$\mathfrak{sl}_{2}$-modules, they exhaust all primitive ideals.

\section{Preliminaries}
Throughout this paper, we denote by $\C$, $\C^{*}$, $\Z$, $\Z_{+}$ and $\N$ the
sets of complex numbers, non-zero complex numbers, integers,
non-negative integers and positive integers, respectively.
All vector spaces, in particular, all algebras and modules are defined over the field $\C$.
For an algebra $R$ and elements $x,y$ in $R$,
we denote by $(x)$ (resp. $(x,y)$) the two-sided ideal of $R$ generated by $x$ (resp. $x,y$).

Consider the Lie algebra $\mathfrak{sl}_{2}$ with the standard basis $\{e, h, f\}$ and the Lie bracket
\begin{equation*}
  [e,f]=h,\ \ \ [h,e]=2e,\ \ \ [h,f]=-2f.
\end{equation*}
Let $D:=\C[t]/(t^{2})$ be the algebra of dual numbers. 
Consider the associated {\it Takiff Lie algebra}
 $\mathfrak{L}:=\mathfrak{sl}_{2}\otimes_{\C}D$ with the Lie bracket
\begin{align*}
  [x\otimes t^{i}, y\otimes t^{j}]=[x, y]\otimes t^{i+j},
\end{align*}
where $x, y\in \mathfrak{sl}_{2}$ and $i, j\in \left\{0,1\right\}$,
and the Lie bracket on the right hand side is the usual $\mathfrak{sl}_{2}$-Lie bracket.
We identify $\mathfrak{sl}_{2}$ with the subalgebra $\mathfrak{sl}_{2}\otimes 1\subseteq \L$,
and set
\begin{align*}
  \overline{e}:=e\otimes t, \ \ \overline{f}:=f\otimes t, \ \ \overline{h}:=h\otimes t.
\end{align*}
Let $\overline{\mathfrak{g}}$ be the ideal of $\L$ generated by $\overline{e},\overline{f},\overline{h}$.
Let,  further,  $\tilde{\mathfrak{n}}_{+}$ denote the subalgebra of $\L$ generated by
$e$ and $\overline{e}$.  We denote by $\tilde{\mathfrak{h}}$
the subalgebra of $\L$ generated by $h$ and $\overline{h}$.
Finally, let $\tilde{\mathfrak{n}}_{-}$ be the subalgebra of $\L$ generated by $f, \overline{f}$.
Then we have the following triangular decomposition of $\L:$
\begin{align*}
  \L=\tilde{\mathfrak{n}}_{+}\oplus \tilde{\mathfrak{h}}\oplus \tilde{\mathfrak{n}}_{-}.
\end{align*}

For a Lie algebra $L$, we denote by $U(L)$ the corresponding universal enveloping algebra.
\begin{defi}{\rm A $\L$-module $M$ is called a} {\it weight module} {\rm provided that} $M=\sum_{\lambda\in \C} M_{\lambda}$, where
\begin{align*}
  M_{\lambda}=\left\{ v\in M\,|\, h\cdot v=\lambda v\right\}.
\end{align*}
{\rm The subspace $M_{\lambda}$ is called the} weight space {\rm of $M$ corresponding to the weight $\lambda$.}
\end{defi}
\begin{defi}{\rm Let $A$ be an associative algebra and $M$ be an $A$-module. Then the}
 {\it annihilator} {\rm Ann$_{A}(M)$ of $M$ in $A$ is defined as follows:}
\begin{align*}
  {\rm Ann}_{A}(M)=\{a\in A: a\cdot m=0\ \ \ {\rm for \ all}\ m\in M\}.
\end{align*}
{\rm An ideal $I\subseteq A$ is called} {\it primitive}
{\rm provided that $I=${\rm Ann}$_{A}(V)$, for some simple $A$-module $V$}.
\end{defi}
The center $Z(\L)$ of $U(\L)$ is the polynomial algebra generated by
the following two algebraically independent elements, see \cite{M}:
\begin{eqnarray*}
  C= \overline{h}h+2\overline{f}e+2\overline{e}f,\ \ \ \
\overline{C} = \overline{h}^{2} +4\overline{f}\overline{e}.
\end{eqnarray*}
For $\lambda\in \C$, consider the corresponding
simple highest weight $\mathfrak{sl}_{2}$-module $L(\lambda)$
with highest weight $\lambda$.
Define
\begin{eqnarray*}
  Q(\lambda):=\mbox{Ind}_{\mathfrak{sl}_{2}}^{\L}L(\lambda)
  =U(\L)\otimes_{U(\mathfrak{sl}_{2})}L(\lambda).
\end{eqnarray*}
For an algebra homomorphism $\boldsymbol{\chi}:\C\left[\, \overline{C}\, \right]\longrightarrow\C$,
set $\chi=\boldsymbol{\chi}\left(\overline{C}\right)$ and consider the quotient module
\begin{align*}
Q(\lambda,\chi):&=Q(\lambda)/\left(\overline{C}-\chi\right)Q(\lambda).
\end{align*}

\begin{prop}\label{prop1}{\rm(see\cite{K,MM})}.
Let $\chi$ and $\lambda$ be as above.
Then the following statements hold.
\begin{itemize}
\item [\rm(1)]  $Q(\lambda)\cong Q(0)\otimes_{\C} L(\lambda)$,
where we consider $L(\lambda)$ as an $\L$-module with the trivial $\overline{\mathfrak{g}}$-action,
and the tensor product is the usual tensor product of $\L$-modules.

\item [\rm(2)] $Q(\lambda,\chi)\cong Q(0,\chi)\otimes_{\C} L(\lambda)$.

\item [\rm(3)] $\displaystyle Q(0,\chi)\cong \bigoplus_{k\geq0}L(2k)$, as an $\mathfrak{sl}_{2}$-module.
\end{itemize}
\end{prop}
The Verma modules for the Takiff $\mathfrak{sl}_{2}$ are studied in detail in \cite{MS}.
For a weight $\boldsymbol{\mu}\in \tilde{\mathfrak{h}}^{*}$,
we set $\mu:=\boldsymbol{\mu}(h)$ and $\overline{\mu}:=\boldsymbol{\mu}(\overline{h})$.
 Using this convention, we have the following proposition.
\begin{prop}\label{propa1}{\rm(see \cite{MS,BJW}).} The Verma $\L$-module $M(\mu,\overline{\mu})$ is simple if and only if $\overline{\mu}\neq0$.
\end{prop}
Note that $C$ and $\overline{C}$ act on the Verma module $M({\mu},\overline{\mu})$
as the scalars $\overline{\mu}({\mu}+2)$ and $\overline{\mu}^{2}$, respectively.
In particular, in this case we have $\chi=\overline{\mu}^{2}$.
\begin{prop}\label{propa3}{\rm (see \cite{VM})}.
Let $\gamma\in \C$ and
$\mathtt{C}:=(h+1)^{2}+4fe$ be the Casimir element for $\mathfrak{sl}_{2}$.
Denote by $J_{n}$ the annihilator of the simple finite dimensional
$\mathfrak{sl}_{2}$-module $L(n)$, $n\in \Z_+$.
Then the following assertions hold.
\begin{itemize}
\item [\rm (a)] For $\gamma\in\C$,
the annihilator {\rm Ann}$_{U(\mathfrak{sl}_{2})}V(\gamma)$ of the $\mathfrak{sl}_{2}$-Verma module $V(\gamma)$
with highest weight $\gamma$
  in $U(\mathfrak{sl}_{2})$ is the two-sided ideal $I_\gamma$ of $U(\mathfrak{sl}_{2})$
  generated by the element $\mathtt{C}-(\gamma+1)^{2}$.
\item [\rm (b)] Every primitive ideal of $U(\mathfrak{sl}_{2})$ coincides with one of the ideals of $I_{\gamma}$, where $\gamma\in \C$,
    or with one of the ideals $J_{n}$, where $n\in \Z_+$. All these ideals are pairwise different.
\end{itemize}
\end{prop}
\section{The module $Q(\lambda)$ and its quotients}
\subsection{Irreducibility of $Q(\lambda,\chi,\theta)$}\label{section 3.1}
For any $\theta\in \C$, consider the module
$Q(\lambda,\chi,\theta):=Q(\lambda,\chi)/(C-\theta)Q(\lambda,\chi)$.

\begin{lemm}\label{lemma2}
Assume that $\lambda\in\C\setminus\Z$.
\begin{description}
  \item[(a)] We  have $Q(\lambda,\chi,\theta)\cong\bigoplus_{k\in \Z}L(\lambda+2k)$, as an $\mathfrak{sl}_{2}$-module.
  \item[(b)] The module $Q(\lambda,\chi)$ is free as a $\C[C]$-module.
  \item[(c)] The module $Q(\lambda)$ is free as a $\C[\,C,\overline{C}\,]$-module.
\end{description}
\end{lemm}

\begin{proof}
Recall from Proposition \ref{prop1} that $Q(\lambda,\chi)\cong Q(0,\chi)\otimes_\C L(\lambda)$
and that
\begin{displaymath}
Q(0,\chi)\cong \bigoplus_{k\in \Z_+}L(2k),
\end{displaymath}
when considered as an $\mathfrak{sl}_2$-module.
For $k\in \Z_+$, consider the following subspace of $Q(0,\chi)$:
\begin{displaymath}
A_{2k}:= \bigoplus_{j\in \Z_+}L(2k+2j).
\end{displaymath}
This gives the following descending filtration on $Q(0,\chi)$:
\begin{equation}\label{eq-1n}
Q(0,\chi)= A_{0}\supset A_2\supset A_4\supset\dots.
\end{equation}
Here  $A_{2k}/A_{2k+2}\cong L(2k)$,  for any $k\in\Z_+$, and also
$\displaystyle \bigcap_{k\in\Z_+}A_{2k}=\varnothing$. In fact, as
a basis of $A_{2k}$ inside $Q(0,\chi)$ we can take the set of all monomials
in $\overline{e}$, $\overline{f}$ and $\overline{h}$ of total
degree at least $k$ in which the degree of $\overline{h}$
is at  most one (we identify these monomials with their images when
applied to the generating element $1\otimes  1$ of $Q(0,\chi)$).
We will call such monomials {\em nice} in this proof.

Since we have assumed that $\lambda\in\C\setminus\Z$, we have
\begin{displaymath}
L(\lambda)\otimes_\C  L(2k)\cong
L(\lambda+2k)\oplus
L(\lambda+2k-2)\oplus\dots
\oplus L(\lambda-2k).
\end{displaymath}
This implies that, if we consider $Q(0,\chi)\otimes_\C L(\lambda)$ as an
$\mathfrak{sl}_2$-module, it decomposes into a direct sum of
$L(\lambda+2k)$, where $k\in\Z$, each appearing with infinite multiplicity.
Let us denote by $X({\lambda+2k})$ the $\mathfrak{sl}_2$-isotypic component of
$Q(\lambda,\chi)$ corresponding to $L(\lambda+2k)$. Now we can consider
inside $X({\lambda+2k})$ the weight space $X({\lambda+2k})_{\lambda+2k}$
of all highest weight vectors. It is infinite  dimensional.

Applying ${}_-\otimes_\C L(\lambda)$  to the filtration \eqref{eq-1n}
and intersecting the outcome with $X({\lambda+2k})_{\lambda+2k}$
induces, after a finite offset, a filtration
\begin{displaymath}
X({\lambda+2k})_{\lambda+2k}=B_0\supset B_2\supset B_4\supset\dots
\end{displaymath}
such that  $B_{2j}/B_{2j+2}$ has dimension one,  for any $j\in\Z_+$, and also
$\displaystyle \bigcap_{k\in\Z_+}B_{2k}=\varnothing$.
We can pick a non-zero element $w^{(k)}_0$ in $B_{0}\setminus B_{2}$.

From our description of a basis in $A_{2k}$, it follows that
$w^{(k)}_0$ will contain, with a non-zero coefficient,
an element of the form $u\otimes v$, where $u$ is a nice monomial
of total degree exactly $k$ and that $w^{(k)}_0$ does not contain
any element in which we find nice monomials of smaller degree.
Since $\overline{\mathfrak{g}}L(\lambda)=0$ and all
monomials in $C$ contain either
$\overline{e}$, $\overline{f}$ and $\overline{h}$ (with total degree one),
it follows that, applying $C$ to a basis element
$\overline{f}^{i}\overline{h}^{\epsilon}\overline{e}^{j}\otimes v$,
we will get three different terms with non-zero coefficients,
the power of $\overline{f}$,
$\overline{h}$ and $\overline{e}$ are greater by one than before, respectively.
Thus applying $C$ to $w^{(k)}_0$ shifts
this minimal degree from $k$  to $k+1$, and it is easy to see
that the outcome will  be non-zero.
Another application of $C$ will move the degree to $k+2$ and so on.
Consequently, it follows that the elements
$\{C^j w^{(k)}_0\,:\, j\in\Z_+\}$  are linearly independent.
Moreover, it also follows that each $C^j w^{(k)}_0$
gives rise to a basis in $B_{2j}/B_{2j+2}$.
This means exactly that the action of $\C[C]$ on
$X({\lambda+2k})_{\lambda+2k}$ is free (in fact, of rank one with basis $w^{(k)}_0$).

Since $C$ is central and $X({\lambda+2k})_{\lambda+2k}$
generates $X({\lambda+2k})$, it follows that
action of $\C[C]$ on
$X({\lambda+2k})$ is free.
Finally, as $k\in\Z_+$ is arbitrary, we deduce Claim~(b).

Claim~(a) follows from Claim~(b)  since $\C[C]/(C-\theta)$
has dimension one.

From the PBW Theorem, $Q(\lambda)$ is free over $U(\overline{\mathfrak{g}})$.
Clearly, nice monomials form a free basis of $U(\overline{\mathfrak{g}})$
over $\C[\overline{C}]$. Therefore $Q(\lambda)$ is free over
$\C[\overline{C}]$. Since we have defined $Q(\lambda,\chi)$
as the quotient  of $Q(\lambda)$  modulo $(\overline{C}-\chi)Q(\lambda)$
and have just proved that $Q(\lambda,\chi)$ is free over
$\C[C]$, the preimage in $Q(\lambda)$ of any basis of $Q(\lambda,\chi)$ over
$\C[C]$ will be a basis of $Q(\lambda)$ over $\C[C,\overline{C}]$.
This implies Claim~(c) and the proof is completed.
\end{proof}


\begin{prop}\label{prop3.2}
\hspace{1mm}

\begin{enumerate}[(a)]
\item  The $\L$-module $Q(\lambda,\chi,\theta)$ is simple if and only if
$\theta\neq\sqrt{\chi}(\lambda+2k)$, for any $k\in\Z$.
\item If $\chi\neq 0$ and $Q(\lambda,\chi,\theta)$ is not simple,
then it has a unique non-zero proper $\L$-submodule,  in particular,
$Q(\lambda,\chi,\theta)$ has length two.
\end{enumerate}
\end{prop}
\begin{proof}
Assume first that there is a non-zero element in $Q(\lambda,\chi,\theta)$
killed by $\overline{e}$. Clearly, $e$ acts on $Q(\lambda,\chi,\theta)$
locally nilpotently. Since $e$  and $\overline{e}$ commute, it follows
that there is a non-zero $v\in Q(\lambda,\chi,\theta)$ such that
$ev=\overline{e}v=0$. Of course, we may assume that $hv=\mu v$, for some
$\mu=\lambda+2l$, where $l\in\Z$.
We have $e\overline{h}v=-2\overline{e}v=0$, which implies that
$\overline{h}v$ is a scalar multiple of $v$ as $Q(\lambda,\chi,\theta)$
contains  a  unique (up to scalar)  $e$-highest weight vector of highest
weight $\mu$.  Let $\overline{\mu}$  be such that
$\overline{h}v=\overline{\mu}v$. Then, by the universal property of
Verma modules,  we have a non-zero homomorphism from $M(\mu,\overline{\mu})$
to $Q(\lambda,\chi,\theta)$ sending the canonical generator of
$M(\mu,\overline{\mu})$ to $v$. In particular,  $\theta=\overline{\mu}(\mu+2)$
and $\chi=\overline{\mu}^2$.

Therefore, if $\theta\neq\sqrt{\chi}(\lambda+2k)$, for any $k\in\Z$,
the  element $\overline{e}$ must act injectively on $Q(\lambda,\chi,\theta)$.
Note that our assumption on $\chi$ and $\theta$  means that this
central character is not the central character of a highest weight module.
Assume $N\subset Q(\lambda,\chi,\theta)$ is a non-zero proper submodule.
Then it must contain  some $\mathfrak{sl}_2$-module $L(\lambda+2k)$.
As $\overline{e}$ acts injectively, applying it to the highest weight vector in
this $L(\lambda+2k)$ we get the highest weight vectors in all
$L(\lambda+2k+2i)$, where $i\in\Z_{+}$. Therefore $N$ contains all later
$\mathfrak{sl}_2$-submodules as well. Since $N$ is proper, there must exist
a minimal $k$ such that $N$ contains $L(\lambda+2k)$. 
Thus there is a non-zero homomorphism $\phi$ from the Verma module
$M(\lambda+2k-2,\sqrt{\chi})$ to $Q(\lambda,\chi,\theta)/N$, and
\begin{align}\label{eq3.2}
  M(\lambda+2j-2,\sqrt{\chi})/{\rm Ker\phi}\cong {\rm Im}\phi.
\end{align}
If $\chi\neq 0$, $\phi$ is an injection followed from 
Proposition \ref{propa1}. 
Thus, $$M(\lambda+2j-2,\sqrt{\chi})\cong Q(\lambda,\chi,\theta)/N,$$
as they have the same character,
which implies that $Q(\lambda,\chi,\theta)/N$ 
is a simple highest weight module.
Hence, $Q(\lambda,\chi,\theta)$ must have the central 
character of a highest weight module.
This leads to a contradiction.
If $\chi=0$, then applying $C$ on both sides of (\ref{eq3.2}), 
we get $0=\theta$, contradiction.
And in this way
we prove that $Q(\lambda,\chi,\theta)$ is simple under our assumption
that $\theta\neq\sqrt{\chi}(\lambda+2k)$, for any $k\in\Z$.

Next let us assume that $Q(\lambda,\chi,\theta)$ is not simple,
that $\chi\neq 0$, and that $N\subset Q(\lambda,\chi,\theta)$ 
is a non-zero proper submodule containing some $L(\lambda+2k)$.
Applying $\overline{e}$ to the
highest weight vector of the latter we either get that all
$L(\lambda+2k+2i)$, where $i\in\Z_{+}$, belong to $N$ or
that $N$ contains a non-zero highest weight vector
of some weight $\lambda+2k+2i$, where $i\in\Z_{+}$.

In the former case, we have a  short exact sequence
\begin{equation}\label{eq-2n}
0\to N\to  Q(\lambda,\chi,\theta)\to
Q(\lambda,\chi,\theta)/N\to 0,
\end{equation}
where $Q(\lambda,\chi,\theta)/N$ is a highest weight module.
This must be simple as $\chi\neq 0$.
In the latter case, $N$ is a simple highest weight module
by a similar argument.

In both cases, the non-highest weight subquotient is simple
by a similar argument: any proper non-zero submodule of it
necessarily gives rise to a simple highest weight
subquotient. On the one hand, the latter must be simple
as $\chi\neq 0$. On the other hand, it has finite length
as an $\mathfrak{sl}_2$-module, a contradiction.
This proves Claim~(b).

Assume now that $\theta=\sqrt{\chi}(\lambda+2k)$, for some
$k\in\N$. Set $\mu:=\lambda+2k-2$ and $\overline{\mu}=\sqrt{\chi}$.
Consider the Verma module $M(\mu,\overline{\mu})$.
It has the same central  character as $Q(\lambda,\chi,\theta)$
and it also has $L(\lambda)$ as an  $\mathfrak{sl}_2$-subquotient.
Therefore, by the universal property of
$Q(\lambda,\chi,\theta)$, there is a non-zero homomorphism
from $Q(\lambda,\chi,\theta)$ to $M(\mu,\overline{\mu})$.
This homomorphism is, clearly, not injective.
This implies that $Q(\lambda,\chi,\theta)$ is not simple.
Denote by $K(\lambda,\chi,\theta)$ the kernel of this
homomorphism.

Finally,  consider the case $\theta=\sqrt{\chi}(\lambda+2k)$, for some
$-k\in\Z_+$. Set $\mu:=\lambda+2k-2$ and $\overline{\mu}=\sqrt{\chi}$.
Consider the Verma module $M(\mu,\overline{\mu})$ and the
corresponding module $K(\mu,\chi,\theta)$ constructed in the
previous paragraph. It has the same central  character as $Q(\lambda,\chi,\theta)$
and it also has $L(\lambda)$ as an  $\mathfrak{sl}_2$-subquotient.
Therefore, by the universal property of
$Q(\lambda,\chi,\theta)$, there is a non-zero homomorphism
from $Q(\lambda,\chi,\theta)$ to $K(\mu,\chi,\theta)$.
This homomorphism is, clearly, not injective.
This implies that $Q(\lambda,\chi,\theta)$ is not simple
and completes the proof.
\end{proof}
\subsection{Annihilator of $Q(\lambda,\chi,\theta)$}\label{section 3.2}
\begin{theo}\label{lemma1}
We have ${\rm Ann}_{U(\L)}Q(\lambda)=0$. Consequently,
${\rm Ann}_{U(\L)}Q(\lambda,\chi,\theta)$ is generated by
$C-\theta$ and $\overline{C}-\chi$,
for any $\theta, \chi\in \C$.
\end{theo}
\begin{proof}
Denote $v_{\lambda}$ the highest weight vector of $L(\lambda)$.
 Let $\{\overline{f}^{i}\overline{h}^{j}\overline{e}^{k}f^{p}h^{q}e^{l}\,|\,i,j,k,p,q,l\in \Z_{+}\}$
 be the PBW basis of $U(\L)$ and $x$ be an arbitrary non-zero element of $U(\L)$.
Then we can write  
\begin{equation*}
  x=\sum_{i,j,k,p,q,l\in \Z_{+}} c_{i,j,k}^{p,q,l}\,\overline{f}^{i}\overline{h}^{j}\overline{e}^{k}f^{p}h^{q}e^{l},\ \
  \mbox{for\ some}\ \ c_{i,j,k}^{p,q,l}\in \C.
\end{equation*}
Denote the finite and non-empty set $\{(i,j,k,p,q,l)\in \Z_{+}^{6}\,|\, c_{i,j,k}^{p,q,l}\neq 0\}$ by $P_{x}$.
Then we define a total lexicographic order on $P_{x}$ as  follows:
 \begin{eqnarray*}
   (i_{1},i_{2},i_{3},i_{4},i_{5},i_{6})> (i'_{1},i'_{2},i'_{3},i'_{4},i'_{5},i'_{6}) \Longleftrightarrow  i_{1}>i'_{1}\ \ \mbox{or} \ \ i_{1}=i'_{1},\  i_{2}>i'_{2}\  \text{ and so on.}
 \end{eqnarray*}
Let $(i_{0},j_{0},k_{0},p_{0},q_{0},l_{0})$ be the minimum element in $P_{x}$.
 Since $x$ has finitely many non-zero  terms, we can rewrite $x$ as
 \begin{equation*}
  x= \SUM{r=0}{m}\SUM{r'=0}{n}\overline{f}^{i_{0}}\overline{h}^{j_{0}}\overline{e}^{k_{0}}
  y_{r,r'}
  +\SUM{(i,j,k,p,q,l)\in P'_{x}}{}c_{i,j,k}^{p,q,l}\overline{f}^{i}\overline{h}^{j}
  \overline{e}^{k}f^{p}h^{q}e^{l},
 \end{equation*}
 where $y_{r,r'}=f^{p_{0}+r}g_{r,r'}(h)e^{l_{0}+r'}$,
 $g_{r,r'}(h)=\sum_{q_{r,r'}=0}^{n_{r,r'}}d_{q_{r,r'}}h^{q_{r,r'}}\in \C[h]$ for $r> 0$, 
 $g_{0,r'}(h)=\sum_{q_{0,r'}=q_{0}}^{n_{0,r'}}d_{q_{0,r'}}h^{q_{0,r'}}\in \C[h]$, 
 $P'_{x}=P_{x}\backslash \{(i_{0},j_{0},k_{0},p_{0}+r,q_{r,r'},l_{0}+r')\,|\,
q_{0,r'}=q_{0},\cdots,n_{0,r'},q_{r,r'}=0,\cdots,n_{r,r'}, r=1,\cdots,m, r'=0,\cdots,n\}$.
 Next we show that there is a non-zero element $u\in Q(\lambda)$ such that $x\cdot u\neq0$.
 Applying $x$ to $1\otimes f^{l_{0}} v_{\lambda}$, we see that
 \begin{align*}
  x\cdot(1\otimes f^{l_{0}} v_{\lambda})
   &=\SUM{(i,j,k,p,q,l)\in P'_{x}}{}c_{i,j,k}^{p,q,l}\overline{f}^{i}\overline{h}^{j}
  \overline{e}^{k}\otimes f^{p}h^{q}e^{l}f^{l_{0}} v_{\lambda}+
   \SUM{r=0}{m}\SUM{r'=0}{n}\overline{f}^{i_{0}}\overline{h}^{j_{0}}\overline{e}^{k_{0}}
  \otimes y_{r,r'}f^{l_{0}} v_{\lambda}\\
  &= \SUM{(i,j,k,p,q,l)\in P'_{x}}{}c_{i,j,k}^{p,q,l}\overline{f}^{i}\overline{h}^{j}
  \overline{e}^{k}\otimes f^{p}h^{q}e^{l}f^{l_{0}} v_{\lambda}+
   \SUM{r=0}{m}\overline{f}^{i_{0}}\overline{h}^{j_{0}}\overline{e}^{k_{0}}\otimes 
  f^{p_{0}+r}g_{r,0}(h)e^{l_{0}}f^{l_{0}} v_{\lambda}\\
  &=\SUM{(i,j,k,p,q,l)\in P'_{x}}{}c_{i,j,k}^{p,q,l}\overline{f}^{i}\overline{h}^{j}
  \overline{e}^{k}\otimes f^{p}h^{q}e^{l}f^{l_{0}} v_{\lambda}+
  \SUM{r=0}{m}
  \overline{f}^{i_{0}}\overline{h}^{j_{0}}\overline{e}^{k_{0}}\otimes 
   c_{\lambda,l_{0}}g_{r,0}(\lambda)f^{p_{0}+r} v_{\lambda},
 \end{align*}
 where $c_{\lambda,l_{0}}=l_{0}!\prod_{i=1}^{l_{0}}(\lambda-(l_{0}-i))$.
 Observing the elements in $P'_{x}$, we see that if $g_{r,0}(\lambda)\neq0$ for some $r$, 
 then $x\cdot(1\otimes f^{l_{0}} v_{\lambda})\neq0$, as desired.
 Otherwise, considering the action of $y_{r',r'}$ on 
 $f^{l_{0}+s} v_{\lambda}$ for $s\in \N$,
 we get that there exists $s_{0}\in \N$ such that
 $$\SUM{r'=0}{n}y_{r',r'}f^{l_{0}+s_{0}} v_{\lambda}=
 \SUM{r'=0}{n}f^{p_{0}+r'}g_{r',r'}(h)e^{l_{0}+r'}f^{l_{0}+s_{0}} v_{\lambda}\neq0,$$
 as $g_{r',r'}(h)$ has only finitely many zeros, for every $r'=0,\ldots, n$.
 Hence, $x\cdot (1\otimes f^{l_{0}+s_{0}} v_{\lambda})\neq 0$.
 It follows that ${\rm Ann}_{U(\L)}Q(\lambda)=0$ which proves the first claim
 of our theorem.

 The second claim of the Theorem follows from the first claim.
\end{proof}
\section{On primitive ideals of Takiff $\mathfrak{sl}_{2}$}
Recall that $Z(\L)=\C[C,\overline{C}]$ and $\C$ is algebraically closed.
By Dixmier-Schur Lemma, see \cite[Proposition~2.6.8]{Di}, $C,\overline{C}$ act on any simple $\L$-module as scalars.
Thus any primitive ideal of $\L$ contains $C-\theta, \overline{C}-\chi$, for some $\theta,\chi\in \C$. We denote:
\begin{itemize}
\item by $J_{n}$, where $n\in \Z_{+}$,
the annihilator in $U(\mathfrak{sl}_{2})$ of the simple finite dimensional
$\mathfrak{sl}_{2}$-module $L(n)$;
\item by  $I(\chi,\theta)$, where $\theta,\chi\in\C$,
the ideal of $U(\L)$, generated by
$C-\theta$ and $\overline{C}-\chi$;
\item by $\mathfrak{I}(\mu)$, where $\mu\in\C$, the ideal of $U(\L)$
generated by $\mathtt{C}-(\mu+1)^{2}$ and $\overline{\mathfrak{g}}$;
\item by $\mathfrak{F}(n)$, where $n\in\Z_{+}$, the ideal of $U(\L)$
generated by $J_{n}$ and $\overline{\mathfrak{g}}$.
\end{itemize}
\begin{theo}\label{theo4.2}
Let $\theta,\chi\in \C$ such that $\theta\neq0$ or $\chi\neq0$. Then the algebra
$U(\L)/(C-\theta,\overline{C}-\chi)$ is simple.
\end{theo}
\begin{proof}
Consider the set
\begin{displaymath}
 A:=\{\overline{f}^{i}\overline{h}^{\epsilon}f^{j}h^{p}e^{q},\,\,\,\,
\overline{f}^{i}\overline{h}^{\epsilon}\overline{e}^{j}h^{p}e^{q}\,|\,i,j,p,q\in \Z_{+}, \epsilon=0,1\},
\end{displaymath}
we claim that this set gives a basis of $U(\L)/(C-\theta,\overline{C}-\chi)$.
Note that, in the quotient, we have
\begin{equation}\label{eq01}
  2\overline{e}f=\theta-\overline{h}h-2\overline{f}e,\ \
  \overline{h}^{2}=\chi-4\overline{f}\overline{e},
\end{equation}
which implies that $U(\L)/(C-\theta,\overline{C}-\chi)$
is generated by the elements in $A$.
In order to prove linear independence, we need do
some preparation.
\begin{clai}
For any $i,j,p,q\in \Z_{+}$, $\epsilon=0,1$, we have
\begin{eqnarray}&\!\!\!\!\!\!\!\!\!\!\!\!\!\!\!\!\!\!\!\!\!\!\!\!&
[\overline{h},\overline{f}^{i}\overline{h}f^{j}h^{p}e^{q}]
=q\theta\overline{f}^{i}\overline{h}f^{j-1}(h+2)^{p}e^{q-1}
-2(q+j)\overline{f}^{i+1}\overline{h}f^{j-1}h^{p}e^{q}
-q\chi\overline{f}^{i}f^{j-1}(h+2)^{p+1}e^{q-1}
\nonumber\\\!\!\!\!\!\!\!\!\!\!\!\!&\!\!\!\!\!\!\!\!\!\!\!\!\!\!\!&
\ \ \ \ \ \ \ \ \ \ \ \ \ \ \ \ \ \ \ \ 
+2q\theta\overline{f}^{i+1}f^{j-2}(h+2)^{p+1}e^{q-1}
-2q\overline{f}^{i+1}\overline{h}f^{j-2}(h-j+3)(h+2)^{p+1}e^{q-1}
\nonumber\\\!\!\!\!\!\!\!\!\!\!\!\!&\!\!\!\!\!\!\!\!\!\!\!\!\!\!\!&
\ \ \ \ \ \ \ \ \ \ \ \ \ \ \ \ \ \ \ \ 
-4q\overline{f}^{i+2}f^{j-2}h^{p+1}e^{q}
-4q(j-2)\overline{f}^{i+2}f^{j-3}(h-j+3)(h+2)^{p+1}e^{q-1},\label{eq4.4}\\
&\!\!\!\!\!\!\!\!\!\!\!\!\!\!\!\!\!\!\!\!\!\!\!\!&
[\overline{h},\overline{f}^{i}f^{j}h^{p}e^{q}]=
q\theta\overline{f}^{i}f^{j-1}(h+2)^{p}e^{q-1}
-2(q+j)\overline{f}^{i+1}f^{j-1}h^{p}e^{q}
\nonumber\\\!\!\!\!\!\!\!\!\!\!\!\!&\!\!\!\!\!\!\!\!\!\!\!\!\!\!\!&
\ \ \ \ \ \ \ \ \ \ \ \ \ \ \ \ \ \ \ 
-q\overline{f}^{i}\overline{h}f^{j-1}(h+2)^{p+1}e^{q-1}
-2q(j-1)\overline{f}^{i+1}f^{j-2}(h+2)^{p+1}e^{q-1},\label{eq14.4}\\
&\!\!\!\!\!\!\!\!\!\!\!\!\!\!\!\!\!\!\!\!\!\!\!\!&
[h,\overline{f}^{i}\overline{h}^{\epsilon}f^{j}h^{p}e^{q}]
=2(q-i-j)\overline{f}^{i}\overline{h}^{\epsilon}f^{j}h^{p}e^{q},\ \ \ \,
[h,\overline{f}^{i}\overline{h}^{\epsilon}\overline{e}^{j}h^{p}e^{q}]
=2(q+j-i)\overline{f}^{i}\overline{h}^{\epsilon}\overline{e}^{j}h^{p}e^{q},\label{eq114.4}\\
&\!\!\!\!\!\!\!\!\!\!\!\!\!\!\!\!\!\!\!\!\!\!\!\!&
[e, \overline{f}^{i}\overline{h}\overline{e}^{j}]=i\chi\overline{f}^{i-1}\overline{e}^{j}
-(4i+2)\overline{f}^{i}\overline{e}^{j+1},\ \ \ \ \ \ \ 
[e, \overline{f}^{i}\overline{e}^{j}]=i\overline{f}^{i-1}\overline{h}\overline{e}^{j},\label{eq4.6}\\
&\!\!\!\!\!\!\!\!\!\!\!\!\!\!\!\!\!\!\!\!\!\!\!\!&
[f,\overline{h}\overline{e}^{i}]=(4i+2)\overline{f}\overline{e}^{i}-i\chi \overline{e}^{i-1},
\ \ \ \ \ \ \ \ \ \ \ \ \ \ \ \ \,
 [f,\overline{e}^{i}]=-i\overline{h}\overline{e}^{i-1},\label{eq4.7}\\
&\!\!\!\!\!\!\!\!\!\!\!\!\!\!\!\!\!\!\!\!\!\!\!\!&
[\overline{h},\overline{f}^{i}\overline{h}^{\epsilon}\overline{e}^{j}h^{p}e^{q}]
=2q\overline{f}^{i}\overline{h}^{\epsilon}\overline{e}^{j+1}(h+2)^{p}e^{q-1},\label{eq14.5}\\
&\!\!\!\!\!\!\!\!\!\!\!\!\!\!\!\!\!\!\!\!\!\!\!\!&
[\overline{e}, \overline{f}^{i}\overline{h}^{\epsilon}\overline{e}^{j}h^{p}e^{q}]
=\overline{f}^{i}\overline{h}^{\epsilon}\overline{e}^{j+1}\left(h^{p}-(h+2)^{p}\right)e^{q}.
\label{eq4.5}
\end{eqnarray}
\end{clai}
It is easy to show, by induction, that, for any $i\in \Z_{+}$, we have
\begin{eqnarray}&\!\!\!\!\!\!\!\!\!\!\!\!\!\!\!\!\!\!\!\!\!\!\!\!&
h^{i}y=y(h+2)^{i},\ \ \ e^{i}\overline{h}=-2i\,\overline{e}e^{i-1}
+\overline{h}e^{i},\ \ \,ef^{i}=if^{i-1}(h-i+1)+f^{i}e, \label{eq4.0}\\
&\!\!\!\!\!\!\!\!\!\!\!\!\!\!\!\!\!\!\!\!\!\!\!\!&
hy^{i}=y^{i}(h+2i),\ \ \,
f^{i}\overline{h}=2i\overline{f}f^{i-1}+\overline{h}f^{i},
\ \ \ \ hf^{i}=f^{i}(h-2i),\label{eq4.1}
\end{eqnarray}
where $y=\overline{e},e$.
Then it follows from (\ref{eq01}), (\ref{eq4.0}) and (\ref{eq4.1}) that
\begin{eqnarray*}&\!\!\!\!\!\!\!\!\!\!\!\!\!\!\!\!\!\!\!\!\!\!\!\!&
[\overline{h},\overline{f}^{i}\overline{h}f^{j}h^{p}e^{q}]
=\overline{f}^{i}\overline{h}^{2}f^{j}h^{p}e^{q}
-\overline{f}^{i}\overline{h}f^{j}h^{p}e^{q}\overline{h}
=\overline{f}^{i}\overline{h}^{2}f^{j}h^{p}e^{q}
-\overline{f}^{i}\overline{h}f^{j}h^{p}(-2q\overline{e}e^{q-1}+\overline{h}e^{q})
\nonumber\\\!\!\!\!\!\!\!\!\!\!\!\!&\!\!\!\!\!\!\!\!\!\!\!\!\!\!\!&
\ \ \ \ \ \ \ \ \ \ \ \ \ \ \ \ \ \ \,
=\overline{f}^{i}\overline{h}^{2}f^{j}h^{p}e^{q}
+2q\overline{f}^{i}\overline{h}f^{j}\overline{e}(h+2)^{p}e^{q-1}
-\overline{f}^{i}\overline{h}(2j\overline{f}f^{j-1}+\overline{h}f^{j})h^{p}e^{q}
\nonumber\\\!\!\!\!\!\!\!\!\!\!\!\!&\!\!\!\!\!\!\!\!\!\!\!\!\!\!\!&
\ \ \ \ \ \ \ \ \ \ \ \ \ \ \ \ \ \ \,
=q\overline{f}^{i}\overline{h}f^{j-1}\big(\theta-\overline{h}(h+2)-2\overline{f}e\big)
(h+2)^{p}e^{q-1}-2j\overline{f}^{i+1}\overline{h}f^{j-1}h^{p}e^{q}
\nonumber\\\!\!\!\!\!\!\!\!\!\!\!\!&\!\!\!\!\!\!\!\!\!\!\!\!\!\!\!&
\ \ \ \ \ \ \ \ \ \ \ \ \ \ \ \ \ \ \,
=q\theta\overline{f}^{i}\overline{h}f^{j-1}(h+2)^{p}e^{q-1}
-q\overline{f}^{i}\overline{h}f^{j-1}\overline{h}(h+2)^{p+1}e^{q-1}
-2(q+j)\overline{f}^{i+1}\overline{h}f^{j-1}h^{p}e^{q}
\nonumber\\\!\!\!\!\!\!\!\!\!\!\!\!&\!\!\!\!\!\!\!\!\!\!\!\!\!\!\!&
\ \ \ \ \ \ \ \ \ \ \ \ \ \ \ \ \ \ \,
=q\theta\overline{f}^{i}\overline{h}f^{j-1}(h+2)^{p}e^{q-1}
-2q(j-1)\overline{f}^{i+1}\overline{h}f^{j-2}(h+2)^{p+1}e^{q-1}
\nonumber\\\!\!\!\!\!\!\!\!\!\!\!\!&\!\!\!\!\!\!\!\!\!\!\!\!\!\!\!&
\ \ \ \ \ \ \ \ \ \ \ \ \ \ \ \ \ \ \ \ \ 
-q\overline{f}^{i}(\chi-4\overline{f}\overline{e})f^{j-1}(h+2)^{p+1}e^{q-1}
-2(q+j)\overline{f}^{i+1}\overline{h}f^{j-1}h^{p}e^{q}
\nonumber\\\!\!\!\!\!\!\!\!\!\!\!\!&\!\!\!\!\!\!\!\!\!\!\!\!\!\!\!&
\ \ \ \ \ \ \ \ \ \ \ \ \ \ \ \ \ \ \,
=q\theta\overline{f}^{i}\overline{h}f^{j-1}(h+2)^{p}e^{q-1}
-2(q+j)\overline{f}^{i+1}\overline{h}f^{j-1}h^{p}e^{q}
-q\chi\overline{f}^{i}f^{j-1}(h+2)^{p+1}e^{q-1}
\nonumber\\\!\!\!\!\!\!\!\!\!\!\!\!&\!\!\!\!\!\!\!\!\!\!\!\!\!\!\!&
\ \ \ \ \ \ \ \ \ \ \ \ \ \ \ \ \ \ \ \ \ 
-2q(j-1)\overline{f}^{i+1}\overline{h}f^{j-2}(h+2)^{p+1}e^{q-1}
+2q\overline{f}^{i+1}\big(\theta-\overline{h}h-2\overline{f}e\big)
f^{j-2}(h+2)^{p+1}e^{q-1}
\nonumber\\\!\!\!\!\!\!\!\!\!\!\!\!&\!\!\!\!\!\!\!\!\!\!\!\!\!\!\!&
\ \ \ \ \ \ \ \ \ \ \ \ \ \ \ \ \ \ \,
=q\theta\overline{f}^{i}\overline{h}f^{j-1}(h+2)^{p}e^{q-1}
-2(q+j)\overline{f}^{i+1}\overline{h}f^{j-1}h^{p}e^{q}
-q\chi\overline{f}^{i}f^{j-1}(h+2)^{p+1}e^{q-1}
\nonumber\\\!\!\!\!\!\!\!\!\!\!\!\!&\!\!\!\!\!\!\!\!\!\!\!\!\!\!\!&
\ \ \ \ \ \ \ \ \ \ \ \ \ \ \ \ \ \ \ \ \ 
-2q(j-1)\overline{f}^{i+1}\overline{h}f^{j-2}(h+2)^{p+1}e^{q-1}
+2q\theta\overline{f}^{i+1}f^{j-2}(h+2)^{p+1}e^{q-1}
\nonumber\\\!\!\!\!\!\!\!\!\!\!\!\!&\!\!\!\!\!\!\!\!\!\!\!\!\!\!\!&
\ \ \ \ \ \ \ \ \ \ \ \ \ \ \ \ \ \ \ \ \ 
-2q\overline{f}^{i+1}\overline{h}f^{j-2}(h-2j+4)(h+2)^{p+1}e^{q-1}
-4q\overline{f}^{i+2}f^{j-2}h^{p+1}e^{q}
\nonumber\\\!\!\!\!\!\!\!\!\!\!\!\!&\!\!\!\!\!\!\!\!\!\!\!\!\!\!\!&
\ \ \ \ \ \ \ \ \ \ \ \ \ \ \ \ \ \ \ \ \ 
-4q(j-2)\overline{f}^{i+2}f^{j-3}(h-j+3)(h+2)^{p+1}e^{q-1}
\nonumber\\\!\!\!\!\!\!\!\!\!\!\!\!&\!\!\!\!\!\!\!\!\!\!\!\!\!\!\!&
\ \ \ \ \ \ \ \ \ \ \ \ \ \ \ \ \ \ \,
=q\theta\overline{f}^{i}\overline{h}f^{j-1}(h+2)^{p}e^{q-1}
-2(q+j)\overline{f}^{i+1}\overline{h}f^{j-1}h^{p}e^{q}
-q\chi\overline{f}^{i}f^{j-1}(h+2)^{p+1}e^{q-1}
\nonumber\\\!\!\!\!\!\!\!\!\!\!\!\!&\!\!\!\!\!\!\!\!\!\!\!\!\!\!\!&
\ \ \ \ \ \ \ \ \ \ \ \ \ \ \ \ \ \ \ \ \ 
-2q\overline{f}^{i+1}\overline{h}f^{j-2}(h-j+3)(h+2)^{p+1}e^{q-1}
+2q\theta\overline{f}^{i+1}f^{j-2}(h+2)^{p+1}e^{q-1}
\nonumber\\\!\!\!\!\!\!\!\!\!\!\!\!&\!\!\!\!\!\!\!\!\!\!\!\!\!\!\!&
\ \ \ \ \ \ \ \ \ \ \ \ \ \ \ \ \ \ \ \ \ 
-4q\overline{f}^{i+2}f^{j-2}h^{p+1}e^{q}
-4q(j-2)\overline{f}^{i+2}f^{j-3}(h-j+3)(h+2)^{p+1}e^{q-1}.
\end{eqnarray*}
This implies (\ref{eq4.4}). Similarly, using
the relations in (\ref{eq01}), (\ref{eq4.0}) and (\ref{eq4.1}),
we can get the equations in (\ref{eq14.4})-(\ref{eq4.5}).
This proves Claim~1.

Now let us prove that the elements in $A$ are linearly
independent in
$U(\L)/(C-\theta,\overline{C}-\chi)$. Assume that this is not the case.
Then there exist some complex numbers
$K^{p,q}_{i,j,\epsilon}$ and $ L^{r,l}_{m,n,\epsilon}$,
almost all but not all zero, such that
\begin{align*}
\SUM{\substack{\epsilon=0,1,\\ i,j,p,q\in\Z_{+}}}{} K^{p,q}_{i,j,\epsilon}
\overline{f}^{i}\overline{h}^{\epsilon}f^{j}h^{p}e^{q}
+\SUM{\substack{\epsilon=0,1,\\m,n,r,l\in\Z_{+}}}{} L^{r,l}_{m,n,\epsilon}
\overline{f}^{m}\overline{h}^{\epsilon}\overline{e}^{n}h^{r}e^{l}=0.
\end{align*}
Applying the inner derivation $\mbox{ad}_{h}$ on above equation, by (\ref{eq114.4}), 
we obtain \begin{align*}
\SUM{\substack{\epsilon=0,1,\\ i,j,p,q\in\Z_{+}}}{}
 K^{p,q}_{i,j,\epsilon}(q-i-j)
\overline{f}^{i}\overline{h}^{\epsilon}f^{j}h^{p}e^{q}
+\SUM{\substack{\epsilon=0,1,\\m,n,r,l\in\Z_{+}}}{} 
L^{r,l}_{m,n,\epsilon}(l+n-m)
\overline{f}^{m}\overline{h}^{\epsilon}\overline{e}^{n}h^{r}e^{l}
=0.
\end{align*}
Since $\mbox{ad}_{h}$ acts diagonal on $U(\L)$,
without loss of generality,
we may assume $q=i+j$ and $m=l+n$,
which is in accordance with the eigenvalue of $\mbox{ad}_{h}$.
Thus, we can rewrite the above in the following form
\begin{align}\label{eqe4.11}
\SUM{i'=0}{i'_{0}} \SUM{j'=0}{j'_{0}} 
\SUM{p'=0}{p'_{0}} N^{p'}_{i',j'}
\overline{f}^{i'}f^{j'}h^{p'}e^{i'+j'}
+\SUM{i=0}{i_{0}} \SUM{j=0}{j_{0}} 
\SUM{p=0}{p_{0}} K^{p}_{i,j}
\overline{f}^{i}\overline{h} f^{j}h^{p}e^{i+j}
+\SUM{\substack{\epsilon=0,1,\\n,r,l\in\Z_{+}}}{}L^{r,l}_{n,\epsilon}
\overline{f}^{l+n}\overline{h}^{\epsilon}\overline{e}^{n}h^{r}e^{l}=0,
\end{align}
where $N^{p'}_{i',j'}, K^{p}_{i,j}\in \C$, 
$N^{p_{0}'}_{i_{0}',j_{0}'}, K^{p_{0}}_{i_{0},j_{0}}\in \C^{*}$,
$L^{r,l}_{n,\epsilon}\in \C$, and almost all but not all are zero.
Denote the left hand side of the above equation by $y$.
Let $j_{1}={\rm max}\{j'_{0},j_{0}\}$ and 
suppose that
 the highest powers of $h$ which appears
in $(\mbox{ad}_{\overline{h}})^{j_{1}}(y)$ with a non-zero coefficient is $p_{0}$.
Using Formulae (\ref{eq4.4}), (\ref{eq14.4}) and (\ref{eq14.5}),
and applying the inner derivation $\mbox{ad}_{\overline{h}}$ 
to (\ref{eqe4.11}), we get
\begin{eqnarray}\label{eq41}&\!\!\!\!\!\!\!\!\!\!\!\!\!\!\!\!\!\!\!\!\!\!\!\!&
0=\SUM{i'=0}{i'_{0}} \SUM{j'=0}{j'_{0}} 
\SUM{p'=0}{p'_{0}} N^{p'}_{i',j'}
[\overline{h},\overline{f}^{i'}f^{j'}h^{p'}e^{i'+j'}]
+\SUM{i=0}{i_{0}} \SUM{j=0}{j_{0}} 
\SUM{p=0}{p_{0}} K^{p}_{i,j}
[\overline{h},\overline{f}^{i}\overline{h} f^{j}h^{p}e^{i+j}]
+\SUM{\substack{\epsilon=0,1,\\n,r,l\in\Z_{+}}}{}L^{r,l}_{n,\epsilon}
[\overline{h},\overline{f}^{l+n}\overline{h}^{\epsilon}\overline{e}^{n}h^{r}e^{l}]
\nonumber\\\!\!\!\!\!\!\!\!\!\!\!\!&\!\!\!\!\!\!\!\!\!\!\!\!\!\!\!&
\ \ 
=\SUM{i'=0}{i'_{0}} \SUM{j'=0}{j'_{0}} 
\SUM{p'=0}{p'_{0}} N^{p'}_{i',j'}
\left((i'+j')\theta\overline{f}^{i'}f^{j'-1}(h+2)^{p'}e^{i'+j'-1}
-2(i'+2j')\overline{f}^{i'+1}f^{j'-1}h^{p'}e^{i'+j'}\right)
\nonumber\\\!\!\!\!\!\!\!\!\!\!\!\!&\!\!\!\!\!\!\!\!\!\!\!\!\!\!\!&
\ \ \ \ \ 
-\SUM{i'=0}{i'_{0}} \SUM{j'=0}{j'_{0}} 
\SUM{p'=0}{p'_{0}} N^{p'}_{i',j'}
(i'+j')\left(\overline{f}^{i'}\overline{h}f^{j'-1}
+2(j'-1)\overline{f}^{i'+1}f^{j'-2}\right)(h+2)^{p'+1}e^{i'+j'-1}
\nonumber\\\!\!\!\!\!\!\!\!\!\!\!\!&\!\!\!\!\!\!\!\!\!\!\!\!\!\!\!&
\ \ \ \ \ 
+\SUM{i=0}{i_{0}} \SUM{j=0}{j_{0}} 
\SUM{p=0}{p_{0}} K^{p}_{i,j}
\left((i+j)\theta\overline{f}^{i}\overline{h}f^{j-1}(h+2)^{p}e^{i+j-1}
-2(i+2j)\overline{f}^{i+1}\overline{h}f^{j-1}h^{p}e^{i+j}\right)
\nonumber\\\!\!\!\!\!\!\!\!\!\!\!\!&\!\!\!\!\!\!\!\!\!\!\!\!\!\!\!&
\ \ \ \ \ 
-\SUM{i=0}{i_{0}} \SUM{j=0}{j_{0}} 
\SUM{p=0}{p_{0}} K^{p}_{i,j}
(i+j)\chi\overline{f}^{i}f^{j-1}(h+2)^{p+1}e^{i+j-1}
-\SUM{i=0}{i_{0}} \SUM{j=0}{j_{0}} 
\SUM{p=0}{p_{0}} K^{p}_{i,j}
4(i+j)\overline{f}^{i+2}f^{j-2}h^{p+1}e^{i+j}
\nonumber\\\!\!\!\!\!\!\!\!\!\!\!\!&\!\!\!\!\!\!\!\!\!\!\!\!\!\!\!&
\ \ \ \ \ 
-\SUM{i=0}{i_{0}} \SUM{j=0}{j_{0}} 
\SUM{p=0}{p_{0}} K^{p}_{i,j}2(i+j)
\left(\overline{f}^{i+1}\overline{h}f^{j-2}(h-j+3)
-\theta\overline{f}^{i+1}f^{j-2}\right)(h+2)^{p+1}e^{i+j-1}
\nonumber\\\!\!\!\!\!\!\!\!\!\!\!\!&\!\!\!\!\!\!\!\!\!\!\!\!\!\!\!&
\ \ \ \ \ 
-\SUM{i=0}{i_{0}} \SUM{j=0}{j_{0}} 
\SUM{p=0}{p_{0}} K^{p}_{i,j}
4(i+j)(j-2)\overline{f}^{i+2}f^{j-3}(h-j+3)(h+2)^{p+1}e^{i+j-1}
\nonumber\\\!\!\!\!\!\!\!\!\!\!\!\!&\!\!\!\!\!\!\!\!\!\!\!\!\!\!\!&
\ \ \ \ \ 
+\SUM{\substack{\epsilon=0,1,\\n,r,l\in\Z_{+}}}{}L^{r,l}_{n,\epsilon}
2l\overline{f}^{l+n}\overline{h}^{\epsilon}\overline{e}^{n+1}(h+2)^{r}e^{l-1}.
\end{eqnarray}
Consider the sum of the following terms from above
\begin{eqnarray*}&\!\!\!\!\!\!\!\!\!\!\!\!\!\!\!\!\!\!\!\!\!\!\!\!&
X:= N^{p'_{0}}_{i'_{0},j'_{0}}
\left((i'_{0}+j'_{0})\theta\overline{f}^{i'_{0}}f^{j'_{0}-1}(h+2)^{p'_{0}}e^{i'_{0}+j'_{0}-1}
-2(i'_{0}+2j'_{0})\overline{f}^{i'_{0}+1}f^{j'_{0}-1}h^{p'_{0}}e^{i'_{0}+j'_{0}}\right)
\nonumber\\\!\!\!\!\!\!\!\!\!\!\!\!&\!\!\!\!\!\!\!\!\!\!\!\!\!\!\!&
\ \ \ \ \ 
-N^{p_{0}'}_{i_{0}',j'_{0}}
(i'_{0}+j'_{0})\overline{f}^{i'_{0}}\overline{h}f^{j'_{0}-1}
(h+2)^{p'_{0}+1}e^{i'_{0}+j'_{0}-1}
\nonumber\\\!\!\!\!\!\!\!\!\!\!\!\!&\!\!\!\!\!\!\!\!\!\!\!\!\!\!\!&
\ \ \ \ \ 
+ K^{p_{0}}_{i_{0},j_{0}}
\left((i_{0}+j_{0})\theta\overline{f}^{i_{0}}\overline{h}f^{j_{0}-1}(h+2)^{p_{0}}e^{i_{0}+j_{0}-1}
-2(i_{0}+2j_{0})\overline{f}^{i_{0}+1}\overline{h}f^{j_{0}-1}h^{p_{0}}e^{i_{0}+j_{0}}\right)
\nonumber\\\!\!\!\!\!\!\!\!\!\!\!\!&\!\!\!\!\!\!\!\!\!\!\!\!\!\!\!&
\ \ \ \ \ 
-K^{p_{0}}_{i_{0},j_{0}}
(i_{0}+j_{0})\chi\overline{f}^{i_{0}}f^{j_{0}-1}(h+2)^{p_{0}+1 }e^{i_{0}+j_{0}-1}.
\end{eqnarray*}
Comparing $j'_{0}$ with $j_{0}$ and $i'_{0}$ with $i_{0}$,
we see that $X\neq 0$.
Therefore, the right hand side of (\ref{eq41}) is nonzero
and the powers of $f$ decrease at least by one.
Then, applying the inner derivation
$(\mbox{ad}_{\overline{e}})^{p_{0}}(\mbox{ad}_{\overline{h}})^{j_{1}}$ on (\ref{eqe4.11})
and using (\ref{eq4.4}), (\ref{eq14.4}), (\ref{eq14.5}) and (\ref{eq4.5}),
we see that the powers of $f$ and $h$ in $y$ are zero.
Thus, we can write$(\mbox{ad}_{\overline{e}})^{p_{0}}(\mbox{ad}_{\overline{h}})^{j_{1}}(y)$ as
\begin{align}\label{eq4.9}
\SUM{\substack{\epsilon=0,1,\\s,r\in\Z_{+}}}{} K^{\epsilon}_{r,s}\overline{f}^{r+s}
\overline{h}^{\epsilon}\overline{e}^{s}e^{r}=0,
\end{align}
where $K^{\epsilon}_{r,s}\in \C$, for $\epsilon=0,1$ and $s,r\in\Z_{+}$,
and again, almost all but not all are zero.
Let $r_{0}$ be the highest power of $e$ above.
Then, applying the inner derivation $(\mbox{ad}_{\overline{h}})^{r_{0}}$
to (\ref{eq4.9}) and using (\ref{eq14.5}), we obtain
\begin{align}\label{eq4.10}
\SUM{s\in\Z_{+}}{} 2^{r_{0}}r_{0}!K^{1}_{r_{0},s}
\overline{f}^{r_{0}+s}\overline{h}\overline{e}^{s+r_{0}}
+\SUM{s'\in\Z_{+}}{} 2^{r_{0}}r_{0}!K^{0}_{r_{0},s'}
\overline{f}^{r_{0}+s'}\overline{e}^{s'+r_{0}}=0.
\end{align}
Denote by $y_{1}$ the left hand side in (\ref{eq4.10}).
Then, according to (\ref{eq4.6}),
 there exists a positive integer $l_{0}$ such that
\begin{align*}
 0= (\mbox{ad}_{e})^{l_{0}}(y_{1})=\SUM{i=0}{n}a_{i}\overline{e}^{i},
 \ \ \ \mbox{for\ some\ \ } a_{i}\in \C, \ a_{n}\in \C^{*}.
\end{align*}
Using the action of $\mbox{ad}_h$,
it follows from (\ref{eq114.4}) that $\SUM{i=0}{n}2ia_{i}\overline{e}^{i}=0$,
which implies that $a_{n}=0$, a contradiction.
Therefore, the elements in $A$ are linearly independent
and hence give rise to a basis in $U(\L)/(C-\theta,\overline{C}-\chi)$.

Let now $\mathcal{S}$ be a non-zero ideal of $U(\L)/(C-\theta,\overline{C}-\chi)$,
and $x$ a non-zero element of $\mathcal{S}$.
We can write $x$ as a finite linear combination of the elements of
the form
$\overline{f}^{i}\overline{h}^{\epsilon}f^{j}h^{p}e^{q}$
and $\overline{f}^{i}\overline{h}^{\epsilon}\overline{e}^{j}h^{p}e^{q}$.
By a similar arguments as above,
one can find $r_{i}\in \Z_{+}$ for $i=0,1,2,3$, such that
$$x_{1}:=(\mbox{ad}_{e})^{r_{0}}(\mbox{ad}_{\overline{h}})^{r_{1}}
(\mbox{ad}_{\overline{e}})^{r_{2}}(\mbox{ad}_{\overline{h}})^{r_{3}}(x)
=\SUM{s=0}{m}c_{s}\overline{e}^{s}\in \mathcal{S},$$
where $c_{s}\in \C$ and $c_{m}\neq 0$.
Letting $\mbox{ad}_h$ act on $x_{1}$ and using the second equation in (\ref{eq114.4}),
it follows that
\begin{align*}
\mbox{ad}_{h}(x_{1})-2mx_{1}
&=\mbox{ad}_{h}\left(\SUM{s=0}{m}c_{s}\overline{e}^{s}\right)-2mx_{1}
=\SUM{s=0}{m}c_{s}(2s-2m)\overline{e}^{s}=\SUM{s=0}{m-1}c_{s}2(s-m)\overline{e}^{s}\in \mathcal{S}.
\end{align*}
Repeating this process finitely many times,
it follows that $\overline{e}^{s_{0}}\in \mathcal{S}$, for some $s_{0}\in \N$.
Note that from the second formula in (\ref{eq4.7}),
we have that $\overline{h}\overline{e}^{s_{0}-1}\in \mathcal{S}$.
Now we consider the following two cases.

\noindent{\bf Case 1:} $\chi\neq 0$.
According to (\ref{eq01}), we obtain
$$\chi \overline{e}^{s_{0}-1}=\overline{h}^{2}\overline{e}^{s_{0}-1}+4\overline{f}\overline{e}^{s_{0}}\in \mathcal{S},$$
which implies that $\overline{e}^{s_{0}-1}\in \mathcal{S}$, since $\chi\neq 0$.
Consequently, $1\in \mathcal{S}$.

\noindent{\bf Case 2:} $\chi=0$, $\theta\neq 0$.
Then, using the first formula in (\ref{eq4.7}),
we see that $\overline{f}\overline{e}^{s_{0}-1}\in \mathcal{S}$.
Using the first relation in (\ref{eq01}), we obtain
\begin{align*}
f\overline{e}^{s_{0}}
&=\frac{1}{2}\left(\theta-\overline{h}(h+2)
-2\overline{f}e\right)\overline{e}^{s_{0}-1}
=\frac{1}{2}\theta\overline{e}^{s_{0}-1}
-\frac{1}{2}\overline{h}\overline{e}^{s_{0}-1}(h+2s_{0})
-\overline{f}\overline{e}^{s_{0}-1}e,
\end{align*}
which implies that $\overline{e}^{s_{0}-1}\in \mathcal{S}$,
since $\theta\neq 0$.
Consequently, $1\in \mathcal{S}$.
\end{proof}

The algebra $U(\L)/(C,\overline{C})$ is, clearly, not simple.
For example, it is infinite dimensional but it 
has infinitely many pairwise non-isomorphic finite
dimensional simple modules $L(n)$, where $n\in\N$, and also
any simple infinite dimensional $\mathfrak{sl}_2$-module is a
simple module over this algebra and has a non-zero annihilator in it.
We can now sum up the above into the following
main result of this paper.

\begin{theo}\label{theo4.3}Let $V$ be a simple $\L$-module such that
$C$ acts on $V$ as $\theta$ and  $\overline{C}$ acts on $V$ as $\chi$,
for some $\theta,\chi\in \C$.
\begin{description}
\item[(a)] If $\theta\neq0$ or $\chi\neq0$, then ${\rm Ann}_{U(\L)}V=I(\chi,\theta)$.
\item[(b)] If $\theta=\chi=0$ and $V$ is simple as an $\mathfrak{sl}_{2}$-module,
           then ${\rm Ann}_{U(\L)}V$ is either $\mathfrak{I}(\mu)$,
           for some $\mu\in\C$, or $\mathfrak{F}(n)$, for some $n\in\Z_+$.
\item[(c)] $I(0,0)$ is a primitive ideal of $U(\L)$.
\end{description}
\end{theo}
\begin{proof}
Claim~(a) follows directly from Theorem~\ref{theo4.2}.
Claim~(b) follows from Proposition \ref{propa3} and the definitions.

To prove Claim~(c), we need to construct a simple
$\L$-module with annihilator $I(0,0)$.
For $a,b,c\in \C$, denote by $C_{a,b,c}$
the one-dimensional $\overline{\mathfrak{g}}$-module
such that $\overline{f}\cdot 1=a$, $\overline{h}\cdot 1=b$, $\overline{e}\cdot 1=c$.
Define $M(a,b,c):={\rm Ind}_{\overline{\mathfrak{g}}}^{\L}C_{a,b,c}
=U(\L)\otimes_{U(\overline{\mathfrak{g}})}C_{a,b,c}$.
Then $M(a,b,c)$ is a $U(\L)$-module.
If we take $b^{2}+4ac=0$, then $\overline{C}\cdot M(a,b,c)=0$.
Let $N_{a,b,c}:= M(a,b,c)/(bh+2ae+2cf)M(a,b,c)$.
Then $C\cdot N_{a,b,c}=0$.

Assume that $b^{2}+4ac=0$ but at least one of $a$, $b$ and $c$ is
not zero. Then we claim that $N_{a,b,c}$ is a simple
$U(\L)$-module with annihilator $I(0,0)$. This certainly
implies Claim~(c). Note that, as at least one of $a$, $b$ and $c$ is
not zero, ${\rm Ann}_{U(\L)}N_{a,b,c}$ cannot coincide
with $\mathfrak{I}(\mu)$ or $\mathfrak{F}(n)$.

Now we give the proof of the claim.
Noting that, by construction, we have $M(a,b,c)\cong 
U(\mathfrak{sl_{2}})$ as vector spaces. 
By our assumptions, we see that $a\neq0$, $b\neq0$, $c\neq0$,
which implies that $e1_{N}=-\frac{1}{2a}(bh+2cf)1_{N}$,
where $1_{N}$ is the generator of $N_{a,b,c}$.
Thus, we have $N_{a,b,c}\cong U(f,h)$ as vector spaces,
where $U(f,h)$ is the subalgebra of $U(\mathfrak{sl_{2}})$ generated by $f,h$.
Let $W$ be a non-zero subalgebra of $U(f,h)$ and $0\neq w\in W$.
Then we can write $$x=\SUM{i=0}{m}\SUM{j=0}{n}d_{i,j}h^{i}f^{j}$$
for some $d_{i,j}\in \C$ and $d_{m,n}\neq0$. 
Applying $(\overline{h}-b)^{n}$ to $x$, we obtain
\begin{align*}
(\overline{h}-b)^{n}x &=\SUM{i=0}{m}\SUM{j=0}{n}d_{i,j}(\overline{h}-b)^{n}h^{i}f^{j}
=\SUM{i=0}{m}\SUM{j=0}{n}d_{i,j}h^{i}(\overline{h}-b)^{n}f^{j} \\
&=\SUM{i=0}{m}\SUM{j=1}{n}d_{i,j}h^{i}(\overline{h}-b)^{n-1}(-2jaf^{j-1}
+bf^{j}-bf^{j}) \\
&=\SUM{i=0}{m}\SUM{j=1}{n}(-2a)d_{i,j}jh^{i}(\overline{h}-b)^{n-1}f^{j-1}\\
&=\SUM{i=0}{m}(-2a)^{n}d_{i,n}n!h^{i}\in W.
\end{align*}
Then, applying $\left(\overline{f}-a\right)$ to $\left(\overline{h}-b\right)^{n}x$, 
we get
\begin{align*}
\left(\overline{f}-a\right)(\overline{h}-b)^{n}x 
&=\SUM{i=0}{m}(-2a)^{n}d_{i,n}n!\left(\overline{f}-a\right)h^{i}
=\SUM{i=0}{m}(-2a)^{n}d_{i,n}n!a((h+2)^{i}-h^{i})\in W.
\end{align*}
From the above, we see that the degree of 
the element $\left(\overline{f}-a\right)(\overline{h}-b)^{n}x$
decreased by one compared to the degree of $(\overline{h}-b)^{n}x$.
Hence, applying $\left(\overline{f}-a\right)^{m}$ on $(\overline{h}-b)^{n}x$,
we get $1_{N}\in W$, which implies that $U(f,h)$ is simple,
and then $N_{a,b,c}$ is simple.

Let $I$ be the annihilator of $N_{a,b,c}$, 
and $J$ be the ideal of $U(\L)$ generated by $C,\overline{C}$.
Denote by $\mathbf{S}$ the subalgebra of $U(\L)$
generated by $f,h,\overline{f},\overline{h}$.
Then $J\subseteq I$, $\mathbf{S}=U(f,h)\otimes \C[\overline{f}-a,\overline{h}-b]$.
Applying $(\overline{f}-a)^{i}$ to $h^{j}$ and $(\overline{h}-a)^{i}$ to $f^{j}$, 
respectively, we get
\begin{align}
 \left(\overline{f}-a\right)^{i}h^{j} &=\left\{
                             \begin{array}{ll}
                                 (2a)^{i}i!,  & \hbox{$i=j$,} \\
                                  0, & \hbox{$i>j$,}
                           \end{array}
                              \right. \label{eq4.11}\\
  \left(\overline{h}-b\right)^{i}f^{j} &=\left\{
                                \begin{array}{ll}
                                  (-2a)^{i}i!, & \hbox{$i=j$,} \\
                                  0, & \hbox{$i>j$.}
                                \end{array}
                             \right. \label{eq4.12}
\end{align}
Let $\mathbf{s}$ be a linear combination of elements in $\mathbf{S}$,
which we write as 
$$\mathbf{s}=\SUM{i=0}{m}\SUM{j=0}{n}q_{i,j}(f,h)\otimes 
\big(\overline{f}-a\big)^{i}\big(\overline{h}-b\big)^{j}.$$
Then we  define a total lexicographical order on $\Z_{+}^{2}$ as follows:
$$(i,j)>(i',j')\Leftrightarrow i> i'\ \ {\rm or\ \ }i=i', j>j'.$$
Thus, there exists the minimal element in the set 
$$\{(i,j)\,|\,q_{i,j}(f,h)\neq 0;\,\,i=0,\cdots, m, j=0,\cdots,n\}.$$
We denote it by $(i_{0},j_{0})$.
It follows from (\ref{eq4.11}) and (\ref{eq4.12}) that 
\begin{align*}
\mathbf{s}\left(f^{i_{0}}h^{j_{0}}\right)
=\SUM{i=0}{m}\SUM{j=0}{n}q_{i,j}(f,h)\otimes 
\big(\overline{f}-a\big)^{i}\big(\overline{h}-b\big)^{j}\left(f^{j_{0}}h^{i_{0}}\right)
=(-1)^{j_{0}}(2a)^{i_{0}+j_{0}}j_{0}!i_{0}!q_{i_{0},j_{0}}(f,h), 
\end{align*}
which implies that $I \cap \mathbf{S}=\{0\}$. 
Note that the Gelfand-Kirillov (GK) dimension of $U(\L)$ is 6,
 the quotient $U(\L)/J$ is a domain and $U(\L)$ is 
free as a module over its center (see \cite{FO,G,G2}).
Thus the GK dimension of $U(\L)/J$ is 4. 
If $J\neq I$, it follows that the GK dimension of $U(\L)/I$ is at most 3, 
which contradicts $I \cap \mathbf{S}=\{0\}$.
Therefore, $J= I$. This completes the proof.
\end{proof}
In Theorem \ref{theo4.3}, we described some primitive ideals of $U(\L)$.
We make the following conjecture:

\noindent{\bf Conjecture}.
Every primitive ideal of $U(\L)$ coincides with either
$\mathfrak{I}(\mu)$, for $\mu\in \C$, or $\mathfrak{F}(n)$, where
$n\in \Z_{+}$, or $I(\chi,\theta)$, where $\chi,\theta\in \C$.

\section*{Acknowledgments}
The author is very grateful to Professor Volodymyr Mazorchuk for his
stimulating discussions.

 \end{document}